\documentclass{amsart}   	% use "amsart" instead of "article" for AMSLaTeX format                		% See geometry.pdf to learn the layout options. There are lots.
\usepackage{parskip}
\usepackage{amssymb,amsthm}
\usepackage{amsthm}
\usepackage{graphicx}
\usepackage{enumitem}
\usepackage{tikz}
\usepackage{tikz-cd}
\usetikzlibrary{calc}
\usetikzlibrary{decorations.pathmorphing}
\usepackage{hyperref}
\usepackage{adjustbox}
\usepackage{mathtools}
\usepackage{mathabx}
\usepackage[all]{xy}

\tikzset{curve/.style={settings={#1},to path={(\tikztostart)
    .. controls ($(\tikztostart)!\pv{pos}!(\tikztotarget)!\pv{height}!270:(\tikztotarget)$)
    and ($(\tikztostart)!1-\pv{pos}!(\tikztotarget)!\pv{height}!270:(\tikztotarget)$)
    .. (\tikztotarget)\tikztonodes}},
    settings/.code={\tikzset{quiver/.cd,#1}
        \def\pv##1{\pgfkeysvalueof{/tikz/quiver/##1}}},
    quiver/.cd,pos/.initial=0.35,height/.initial=0}

\tikzset{tail reversed/.code={\pgfsetarrowsstart{tikzcd to}}}
\tikzset{2tail/.code={\pgfsetarrowsstart{Implies[reversed]}}}
\tikzset{2tail reversed/.code={\pgfsetarrowsstart{Implies}}}
\tikzset{no body/.style={/tikz/dash pattern=on 0 off 1mm}}

\graphicspath{ {./images/} }
\newtheorem{thm}{Theorem}[section]
\newtheorem{prop}[thm]{Proposition}
\newtheorem{lem}[thm]{Lemma}
\newtheorem{corollary}[thm]{Corollary}

\theoremstyle{definition}

\newtheorem{ex}[thm]{Example}

\newcommand{\uxa}{\ensuremath{(\underline{X},\underline{A})}}

\newcommand{\ux}{\ensuremath{(\underline{X},\underline{\ast})}} 
 
\newcommand{\caa}{\ensuremath{(\underline{CA},\underline{A})}}

\newcommand{\clxx}{\ensuremath{(\underline{C\Omega X},\underline{\Omega X})}}

\newcommand{\kl}{\ensuremath{(\underline{K},\underline{L})}}

\newcommand{\cal}[1]{\ensuremath{\mathcal{#1}}}

\newcommand{\lk}{\ensuremath{\mathrm{lk}_{\cal{K}}(i)}}
\usepackage[utf8]{inputenc}

\title{Loop spaces of polyhedral products associated with the polyhedral join product}
\author{Briony Eldridge}
\date{May 2024}

\begin{document}
\begin{abstract}
    We give a homotopy equivalence for the loop space of the moment-angle complex associated with a simplicial complex formed by the polyhedral join operation, and give necessary conditions for this loop space to be a finite type product of spheres and loops on spheres.
\end{abstract}
\maketitle

\section{Introduction}
Polyhedral products are topological spaces formed by attaching products of spaces together according to an underlying simplicial complex. The study of polyhedral products is of interest across many different disciplines of mathematics, including toric topology, combinatorics, geometric group theory, and complex geometry \cite{bahri2024polyhedralproductsfeatureshomotopy}. Specifically, let $K$ be a simplicial complex on $m$ vertices. For $1 \leq i \leq m$, let $(X_i,A_i)$ be a pair of pointed $CW$-complexes, where $A_i$ is a pointed $CW$ subcomplex of $X_i$. Let $(\underline{X}, \underline{A})= \{(X_i,A_i)\}^{m}_{i=1}$ be the sequence of $CW$-pairs. We can identify each simplex $\sigma \in K$ with the sequence $(i_1,\cdots,i_k)$ where $i_1,\cdots,i_k$ are the vertices of $\sigma$ and $i_1 < \cdots < i_k$. Define $(\underline{X},\underline{A})^{\sigma}$ as \[(\underline{X},\underline{A})^{\sigma} = \prod^{m}_{i=1} Y_i, \;\; Y_i = \begin{cases} X_i & i \in \sigma \\ A_i & \mathrm{otherwise}. \end{cases} \]
The \emph{polyhedral product associated to K} is defined as \[(\underline{X},\underline{A})^{K} =\bigcup_{\sigma \in K}  (\underline{X},\underline{A})^{\sigma} \subseteq \prod^{m}_{i=1} X_i.\]An important special case is the \emph{moment-angle complex}, denoted $\cal{Z}_K$, which is the polyhedral product where each pair is $(D^2,S^1)$. There are two problems that provide motivation for this work: how the homotopy type of the polyhedral product changes under simplicial operations, and determining the homotopy type of loop spaces of polyhedral products. 

It is well known that the polyhedral product on the join of two simplicial complexes is the Cartesian product of polyhedral products on each complex, and in \cite{connect} a homotopy equivalence for the loop space of the polyhedral product on the connected sum of simplicial complexes was given. In this paper, we consider how the homotopy theory of the moment-angle complex changes under the \emph{polyhedral join product} operation, which is defined similarly to the polyhedral product operation, but with pairs of spaces $(X_i,A_i)$ and the Cartesian product being replaced with pairs of simplicial complexes $(K_i,L_i)$ and the join operation. To be exact, let $M$ be a simplicial complex on $[m]$. For $1\leq i \leq m$, and let $(\underline{K}, \underline{L})= \{(K_i,L_i)\}^{m}_{i=1}$ be a sequence of pairs of simplicial complexes, where $L_i$ is a subcomplex on $K_i$, considered to be on the vertex set of $K_i$. We can identify each simplex $\sigma \in M$ with the sequence $(i_1,...,i_k)$ where $i_1,...,i_k$ are the vertices of $\sigma$ and $i_1 < \cdots < i_k$. Define $(\underline{K},\underline{L})^{*\sigma}$ as the following subcomplex of $K_1*\cdots*K_m$: 
    \[(\underline{K},\underline{L})^{*\sigma} = \overset{m}{\underset{i=1} \ast} \; Y_i,\;\; Y_i = \begin{cases} K_i & i \in \sigma \\ L_i & \mathrm{otherwise}. \end{cases}\]
    The \emph{polyhedral join product} is defined as 
    \[(\underline{K},\underline{L})^{*M} = \bigcup_{\sigma \in M} (\underline{K},\underline{L})^{*\sigma}.\]
    The polyhedral join product produces a new simplicial complex, but retains combinatorial qualities of the ingredient complexes. Just as the polyhedral product unifies various constructions of spaces across different fields, the polyhedral join product construction unifies various constructions of simplicial complexes. In \cite{Ewald_1986}, the concept of the simplicial wedge construction for toric varieties was introduced, and this construction was first considered in regards to the polyhedral product in \cite{bahri2015operations}. This was then generalised in \cite{construct} to the composition complex (at the time, this was named the substitution complex, but has since changed). The composition complex has applications in topological data analysis \cite{kara2020algebraiccombinatorialpropertiesweighted}. This construction is similar, but different to the substitution complex in \cite{Abramyan_2019}, where it was shown that specific substitution complexes are the smallest simplicial complexes that realise iterated higher Whitehead products. Substitution complexes have also been considered in geometric group theory \cite{MR2961283}. Recently, a specific example of the substitution complex was used to prove a stability theorem for bigraded persistent double homology \cite{bahri2024stabilitytheorembigradedpersistence}. Both the substitution complex and the construction complex are specific examples of polyhedral join products, which was first defined in \cite{construct}. Polyhedral products and polyhedral join products have a nice relation, as characterised in \cite{Vidaurre_2018} (Theorem 2.9). Recently, Steenrod operations with relation to polyhedral join products has been explored \cite{agarwal2024steenrodoperationspolyhedralproducts}.

In this paper, we obtain a homotopy equivalence for $\Omega \uxa^{\kl^{*M}}$. To do this, we utilise facts about the underlying simplicial complex $M$. Recall that the \emph{link} of a vertex is the subcomplex $\mathrm{lk}_M(v) = \{ \tau \in M | \{v\} \cap \tau = \emptyset, \{v\} \cup \tau \in K\}$ and the \emph{restriction} of a vertex is the subcomplex $M \setminus v = \{ \tau \in M | \{v\} \cap \tau = \emptyset\}$.
\begin{thm}\label{mainthrm1}
    Let $\kl^{*M}$ be a polyhedral join product. There is a homotopy equivalence 
     \[\Omega \uxa^{\kl^{*M}} \simeq \prod_{i=1}^m \Omega \uxa^{K_i}  \times \prod_{i=k}^m \Omega (G_i*H_i)\]
     where $G_i$ is the homotopy fibre of the map $\uxa^{\kl^{*\mathrm{lk}_{M\setminus \{m,\cdots,i+1\}}(i)}} \to \uxa^{\kl^{*M \setminus \{m,\cdots,i\}}}$, $H_i$ is the homotopy fibre of the map $\uxa^{L_i} \to \uxa^{K_i}$ and $k \in [m]$ is such that $M \setminus \{m,\cdots,k\}$ is a simplex. 
\end{thm}
 Such a homotopy equivalence allows us to explore the homotopy type of $\Omega \uxa^{\kl^{*M}}$, which is the second problem motivating this work.

Let \cal{P} be the full subcategory of topological spaces that are homotopy equivalent to a finite type product of spheres and loops on spheres. We can then consider when $\Omega \uxa^K$ belongs to  $\cal{P}$. This is part of a wider body of work investigating Anick's conjecture. Anick's conjecture states that if a space $X$ is a finite, connected $CW$-complex, then localised at almost all primes $p$, $\Omega X$ decomposes as a finite type product of spheres, loops on spheres, and indecomposable space related to Moore spaces. Currently, there are two known types of simplicial complexes that give rise to polyhedral products  such that $\Omega \caa^{K} \in \cal{P}$; that is to say, they verify Anick's conjecture without the need to localise. One type is simplicial complexes such that the associated polyhedral product $\caa^K$ belongs to the collection of topological spaces which are homotopy equivalent to a finite type wedge of spheres, denoted $\cal{W}$. Much progress has been made in determining which families of simplicial complexes give rise to such polyhedral products. For example, when $K$ is shifted \cite{grbic2011homotopy}, flag with chordal $1$-skeleton \cite{Grbicflag}, and most generally, a totally fillable complex \cite{IRIYE2013716}, then $\caa^K \in \cal{W}$. If $\caa^K \in \cal{W}$, then the Hilton-Milnor theorem implies that $\Omega \caa^K \in \cal{P}$. Many polyhedral products are not wedges of spheres, but their loop space belongs to $\cal{P}$, and recently much progress in this area has been made. \cite{stanton2024loop} showed that when $K$ is the $k$-skeleton of a flag complex, $\Omega \caa^K \in \cal{P}$ and in\cite{vylegzhanin2024loop} an explicit loop space decomposition was given. Furthermore, in \cite{stanton2024loopspacedecompositionsmomentangle}, it was shown that $\Omega \caa^K \in \cal{P}$ if and only if each full subcomplex $K_I$ with a complete 1-skeleton is such that $\Omega \caa^{K_I} \in \cal{P}$. However, the polyhedral join product produces large simplicial complexes, and it is hard to classify all subcomplexes with a complete $1$-skeleton, and so a different approach is required. For the polyhedral join product, we prove the following.
\begin{thm}\label{mainthrm2}
   Let $M$ be a simplicial complex on $[m]$ vertices. Let $\kl^{*M}$ be a polyhedral join product and let $K_i$ be a simplicial complex on the vertex set $[k_i]$ for $i \in [m]$. Let $\caa$ be a sequence of pairs $(CA_j,A_j)$ where $1 \leq j \leq k_1 + \cdots + k_m$. Suppose $\Sigma A_j \in \cal{W}$ for all  $j$. Let $H_i$ be the homotopy fibre of the map $\caa^{L_i} \to \caa^{K_i}$ for $i \in [m]$. If $\Omega \caa^{K_i} \in \cal{P}$ and $\Sigma H_i \in \cal{W}$ for each $i \in [m]$, then $\Omega \caa^{\kl^{*M}} \in \cal{P}$.
\end{thm}
Theorem \ref{mainthrm2} gives the conditions under which polyhedral products associated with polyhedral join products verify Anick's conjecture; conditions that depend on the pairs $\kl$ and the polyhedral product $\caa^M$. We  use Theorem \ref{mainthrm2} in section 7 to generate new families of simplicial complexes such that $\Omega \cal{Z}_{\kl^{*M}} \in \cal{P}$, but the simplicial complex $\kl^{*M}$ is not the $k$-skeleton of a flag complex, or such that $\cal{Z}_{\kl^{*M}} \in \cal{W}$. To do this, we analyze the substitution complex.  
\par
The structure of this paper is as follows. In Section 2 we cover some preliminary results, and in Section 3 we prove many conbinatorial results concerning the polyhedral join product. In Section 4 we prove a homotopy equivalence concerning polyhedral products, and in Section 5 we apply this to the polyhedral join product. In Section 6 we consider when the polyhedral join product preserves the property of $\Omega \caa^{\kl^{*M}} \in \cal{P}$. Finally, in Section 7, we define new families of simplicial complexes for which  $\Omega \caa^{\kl^{*M}} \in \cal{P}$.
The author would like to thank Stephen Theriault for many valuable and constructive discussions during the preparation of this work. The author would also like to thank Lewis Stanton for reading a draft of this work and providing insightful comments. 
\section{Preliminary results}
\subsection{Properties of \cal{W} and \cal{P}}
Let $\cal{W}$ denote the collection of topological spaces that are homotopy equivalent to a finite type wedge of simply connected spheres, and let $\cal{P}$ be the collection of $H$-spaces which are homotopy equivalent to a finite type product of spheres and loops on simply connected spheres. Relations between spaces in $\cal{P}$ and $\cal{W}$ proved in this section will be used throughout the paper. We begin by stating several results, the first of which follows from the Hilton-Milnor Theorem.

\begin{lem}\label{loopinp}
    If $X \in \cal{W}$, then $\Omega {X} \in \cal{P}. \qed $ 
\end{lem}
The following two lemmas are well known results. For proofs, see \cite{soton467338} and \cite{connected2020amelotte} respectively.
\begin{lem}\label{sw}
    If $X \in \cal{P}$ then $\Sigma X \in \cal{W}. \qed$
\end{lem} 
\begin{lem}\label{retractinw}
    Suppose $X \in \cal{W}$, and let $A$ be a space which retracts off $X$. Then $A \in \cal{W}. \qed$
\end{lem}

 \begin{lem}\cite{stanton2024loop}\label{retractinp}
    Let $X \in \cal{P}$, and let $A$ be a space which retracts off $X$. Then $A \in \cal{P}. \qed$
\end{lem}

\begin{lem}\label{1.3}
    Let $X$ be a pointed $CW$-complex, and suppose $\Omega \Sigma X \in \cal{P}$. Then $\Sigma X \in \cal{W}$. 
\end{lem}

\begin{proof}
    It is well known that if $Y$ is a co-$H$ space, then $Y$ retracts off $\Sigma \Omega Y$. First note that as $\Omega \Sigma X \in \cal{P}$, we obtain $\Sigma \Omega \Sigma X \in \cal{W}$ by Lemma \ref{sw}. Although $X$ may not be a co-$H$ space, $\Sigma X$ is, and so $\Sigma X$ retracts off $\Sigma \Omega \Sigma X$. As $\Sigma X$ retracts off a space in $\cal{W}$, by Lemma \ref{retractinw} $\Sigma X \in \cal{W}$. 
    %Let $\nu: X \to \Omega \Sigma X$ be the suspension map which is adjoint to the identity on $\Sigma X$, and let $ev: \Sigma \Omega \Sigma X \to \Sigma X$ be the evaluation map. Consider $\Sigma \nu : \Sigma X \to \Sigma \Omega \Sigma X$. The composition $\Sigma \nu \circ ev$ is homotopic to the identity on $\Sigma X$. As this map factors through $\Sigma \Omega \Sigma X$, and $\Sigma \Omega \Sigma X \in \cal{W}$, $\Sigma X$ must also be an element of $\cal{W}$. 
\end{proof}

\subsection{Polyhedral products}

In this subsection, we cover some essential results concerning polyhedral products and loop space decompositions of polyhedral products. We begin by stating some fundamental properties of polyhedral products relating to the underlying simplicial complex. Recall that the \emph{join} of two simplicial complexes $K_1$ and $K_2$ is the simplicial complex \[K_1*K_2 = \{ \sigma \cup \tau | \sigma \in K_1\; \mathrm{and} \;\tau \in K_2\}.\] The following lemma follows from this definition and the definition of the polyhedral product.
\begin{lem}\label{join}\cite[Proposition 4.2.5]{buchstaber2014toric}
    Let $K = K_1 * K_2$, where $K_1$ is on the vertex set $\{1,\cdots,l\}$ and $K_2$ is on the vertex set $\{l+1,\cdots,m\}$. Then \[\uxa^{K} = \uxa^{K_1} \times \uxa^{K_2}. \eqno \qed\] 
\end{lem}

Let $L$ be a simplicial complex on the vertex set $[l]$, where $l \leq m$. Let $\overline{L}$ be the simplicial complex on the vertex set $[m]$ by considering the elements of $[m]$ that are not vertices of $L$ as \emph{ghost vertices}.

\begin{prop} \label{funda2}\cite{grbic2011homotopy}
Let $K$ be a simplicial complex on the vertex set $[m]$, and suppose $K = K_1 \cup_L K_2$, that is, $K$ is a pushout of $K_1$ and $K_2$ over $L$. Then, by regarding $\overline{L}$, $\overline{K_1}$, $\overline{K_2}$ as simplicial complexes on the vertex set $[m]$, there is a pushout of polyhedral products 

\[\begin{tikzcd}
\uxa^{\overline{L}} \arrow[d] \arrow[r]
& \uxa^{\overline{K_1}} \arrow[d] \\
\uxa^{\overline{K_2}} \arrow[r]
&\uxa^K.
\end{tikzcd} \eqno \] \qed
\end{prop}

Let $K$ be a simplicial complex on the vertex set $[m]$. For a vertex $v \in K$, the \emph{star}, \emph{restriction} and \emph{link} of $v$ are the subcomplexes
\[\mathrm{st}_K(v) = \{ \tau \in K | \{v\} \cup \tau \in K\},\]
\[K \setminus v = \{ \tau \in K | \{v\} \cap \tau = \emptyset\},\]
\[\mathrm{lk}_K(v) = \{ \tau \in K | \{v\} \cap \tau = \emptyset, \{v\} \cup \tau \in K\} = \mathrm{st}_K(v) \cap K \setminus v. \] 

From the above definitions, for each vertex $i \in [m]$ there is a pushout of simplicial complexes
\begin{equation}\label{1}
    \begin{tikzcd}
    \mathrm{lk}_K(i) \arrow[r] \arrow[d]
    &\mathrm{st}_K(i) \arrow[d] \\
    K \setminus \{i\} \arrow[r]
    & K.
\end{tikzcd}\end{equation}

As $\mathrm{st}_K(i) =  \mathrm{lk}_K(i) * i$, by Lemma \ref{join}, we can rewrite $\uxa^{ \mathrm{st}_K(i)}$ as $\uxa^{ \mathrm{lk}_K(i)} \times X_i$. Applying Proposition \ref{funda2} we obtain the following result.

\begin{corollary}\label{lsrppp}
For each vertex $i \in [m]$ there is a pushout 
    \[\begin{tikzcd}
\uxa^{\lk} \times A_i \arrow[d] \arrow[r]
& \uxa^{\lk} \times X_i \arrow[d] \\
\uxa^{K \setminus i} \times A_i \arrow[r]
&\uxa^K
\end{tikzcd} \]
where $\lk$ is regarded as a simplicial complex on the vertex set of $K \setminus i$. \qed
\end{corollary}

We now give some results concerning a specific type of subcomplex. The subcomplex $L$ is a \emph{full subcomplex} of $K$ if every face in $K$ on the vertex set $[l]$ is also a face of $L$.
We denote a full subcomplex on a set $I \subseteq [m]$ as $K_I$.

\begin{lem}\cite{Denham_2007}\label{retract1}
    Let $K$ be a simplicial complex on the vertex set $[m]$ and let $L$ be a full subcomplex of $K$ on the vertex set $[l]$. Then the induced map of polyhedral products $\uxa^L \to \uxa^K$ has a left inverse \[r: \uxa^K \to \uxa^L. \eqno \qed\] 
\end{lem}

\begin{lem}
    If $\Omega \uxa^K \in \cal{P}$, then for any full subcomplex $K_I$ of $K$, we have $\Omega \uxa^{K_I} \in \cal{P}$.
\end{lem}
\begin{proof}
    By Lemma \ref{retract1} $\Omega \uxa^{K_I}$ is a retract of $\Omega \uxa^K$, and so by Lemma \ref{retractinp} we have $\Omega \uxa^{K_I} \in \cal{P}$. 
\end{proof}

\begin{lem}\label{susfinw}
    Let $L$ be a full subcomplex of $K$, and let $\Omega \uxa^K \in \cal{P}$. If $F$ is the homotopy fibre of  the map $\uxa^L \to \uxa^K$, then $\Sigma F \in \cal{W}$. 
\end{lem}
\begin{proof}
    As $L$ is a full subcomplex of $K$, the map $\uxa^{L} \to \uxa^{K}$ has a left homotopy inverse by Lemma \ref{retract1}, and so the connecting map $\delta: \Omega \uxa^{K} \to F$ has a right homotopy inverse. Therefore the homotopy fibration sequence $\Omega \uxa^{L} \to \Omega \uxa^{K} \xrightarrow{\delta} F$ splits to give a homotopy equivalence $\Omega \uxa^{K} \simeq \Omega \uxa^{L} \times F$. As $\Omega \uxa^{K} \in \cal{P}$, $F$ must therefore be in $\cal{P}$ by Lemma \ref{retractinp}. By Lemma ~\ref{sw}, $\Sigma F$ must therefore be in $\cal{W}$. 
\end{proof}

\begin{lem}\label{fibsplit1}\cite{torichtpytheory}
    Let $G_i$ be the homotopy fibre of the map $\uxa^{\mathrm{lk}_{k\setminus \{m,\cdots,i+1\}}(i)} \to \uxa^{K \setminus \{m,\cdots,i\} }$, and let $Y_i$ be the homotopy fibre of the map $A_i \to X_i$. Then there is a homotopy equivalence 
    \[\Omega \uxa^K \simeq \prod_{i=1}^m \Omega X_i \times \prod_{i=1}^m \Omega (\Sigma G_i \wedge Y_i). \eqno \qed\]
\end{lem}

We can apply this Lemma to determine when $\Sigma^k G_i \in \cal{W}$.

\begin{lem}\label{minsphere}
    Let $K$ be a simplicial complex on the vertex set $[m]$. Let $\Omega \caa^K \in \cal{P}$, and let $G_i$ denote the homotopy fibre of the map $\caa^{\mathrm{lk}_{K}(i)} \to \caa^{K \setminus i}$. Suppose $A_i$ is $(k-1)$-connected. If $\Sigma A_i \in \cal{W}$, then $\Sigma^{k} G_m \in \cal{W}$.
\end{lem}
\begin{proof} 
    By Lemma \ref{fibsplit1} we have that 
    \[\Omega \caa^K \simeq \Omega \caa^{K \setminus i} \times \Omega \Sigma (G_i \wedge A_i).\] As $\Omega \caa^K \in \cal{P}$, by Lemma \ref{retractinp} we have $\Omega (\Sigma G_i \wedge A_i) \in \cal{P}$. By Lemma \ref{1.3}, $\Sigma G_i \wedge A_i \in \cal{W}$. As $\Sigma A_i \in \cal{W}$ by hypothesis, $\Sigma A_i$ is homotopy equivalent to $\bigvee_{j \in \alpha} S^j$ for some indexing set $\alpha$. Therefore \[\Sigma G_i \wedge A_i \simeq \bigvee_{j \in \alpha} G_i \wedge S^j \simeq \bigvee_{j \in \alpha} \Sigma^{j} G_i.\]
    As each $\Sigma^j G_i$ is retracting off $\Sigma G_i \wedge A_i$, we have $\Sigma^j G_i \in \cal{W}$ by Lemma \ref{retractinw}. Furthermore, as $A_i$ is $(k-1)$-connected, $\Sigma A_i$ is $k$-connected. Therefore the smallest $j \in \alpha$ is $k$, and so $\Sigma^k G_i \in \cal{W}$. 
\end{proof}

\begin{lem}\label{conofk}
    Let $K$ be a simplicial complex on the vertex set $[m]$. If $\caa$ is a sequence of pairs of spaces where each $A_i$ is at least $(k-1)$-connected for some $k$, then $\caa^K$ is at least $(2k-2)$-connected.
\end{lem}
\begin{proof}
     After suspending, by \cite{bahri2008polyhedral} there is a homotopy equivalence 
    \[\Sigma \caa^K \simeq \Sigma \bigvee_{I \notin K} |K_I|*\widehat{A}^I\]
    where $|K_I|$ is the geometric realisation of the full subcomplex on the vertex set $I$, and $\widehat{A}^I$ is the iterated smash product $A_{i_1} \wedge \cdots A_{i_I}$. By assumption $\mathrm{conn}(A_i) \geq k-1$ for all $i \in [m]$, and so the connectivity of $\widehat{A}^I$ is at least $k|I| - 1$. Therefore the connectivity of $|K_I|*\widehat{A}^I \simeq \Sigma |K_I| \wedge \widehat{A}^I$ is at least $\mathrm{conn}(|K_I|) + k|I| + 1$. This is minimised when $|I|=2$, in which case $|K_I| \simeq S^0$. Therefore the lower bound for the connectivity of $|K_I|*\widehat{A}^I$ is $2k-1$, and so $\Sigma \caa^K$ is at least $(2k-1)$-connected. After desuspending, we obtain $\caa^K$ is at least $(2k-2)$-connected.
    %The minimal missing face that can appear in any simplicial complex $K$ is two disjoint points. Therefore, the minimal term appearing in this wedge is at least as connected as $\Sigma (S^0 * \widehat{A}^{I'}) \simeq \Sigma^2  \widehat{A}^{I'}$, where $|I'| = 2$. As each $A_i$ is at least $k-1$-connected, $\widehat{A}^{I'}$ will be at least $2k-1$-connected. Therefore, $\Sigma^2  \widehat{A}^{I'}$ is at least $2k+2$-connected. As the smallest missing face that could appear in $K$ is on a vertex set greater than or equal to $I'$, $\Sigma \caa^K$ is at least $2k+2$ connected. After desuspending, we obtain that $\caa^K$ is at least $2k+1$ connected.  
\end{proof}

\begin{corollary}\label{conofl}
    Let $\overline{L}$ be  a simplicial complex on the vertex set $[m]$ with ghost vertices. If $\caa$ is a sequence of pairs of spaces where each $A_i$ is at least $(k-1)$-connected for some $k$, then $\caa^{\overline{L}}$ is at least $(k-1)$-connected. 
\end{corollary}
\begin{proof}
    Note that there is an equality 
    \[\caa^{\overline{L}} = \caa^L \times \prod_{i \notin L} A_i\] that follows from the definition of the polyhedral product. The lower bound for connectivity immediately follows.
\end{proof}

\begin{lem}\label{connected1} 
    Let $K$ be a simplicial complex with no ghost vertices, and let $L$ be a subcomplex of $K$. Let $H$ be the homotopy fibre of the map $\caa^{\overline{L}} \to \caa^K$. Suppose that each $A_i$ is at least $(k-1)$-connected. Then $H$ is at least $(k-1)$-connected. 
\end{lem}
\begin{proof}
    The fibration sequence $H \to \caa^{\overline{L}} \to \caa^K$ induces a long exact sequence of homotopy groups. By Lemma \ref{conofk} and Corollary \ref{conofl} we obtain   $\pi_n(\caa^K) = 0$ for $n \leq 2k-2$ and $\pi_n(\caa^{\overline{L}}) = 0$ for $n < k-1$, we have that $\pi_n(H)=0$ for $n <k-1$. When $n = k-1$, consider the following section of the long exact sequence 
    \[\pi_{k}(\caa^K) \to \pi_{k-1}(H) \to \pi_{k-1}(\caa^{\overline{L}}) \to \pi_{k-1}(\caa^K).\]
    As $\caa^K$ is at least $(2k-2)$-connected by Lemma \ref{conofk}, this splits to give the isomorphism 
    \[\pi_{k-1}(H) \cong \pi_{k-1}(\caa^{\overline{L}}).\] 
    Therefore $\pi_{k-1}(H)$ is precisley as connected as $ \pi_{k-1}(\caa^{\overline{L}})$. By Corollary \ref{conofl}, $\caa^{\overline{L}}$ is at least $(k-1)$-connected, and therefore $H$ is at least $(k-1)$-connected.  
\end{proof}

\subsection{Cube Lemma}
We conclude this section by stating the Cube Lemma, which will be used in later constructions. 
\begin{lem}\cite{Cube}\label{cube}
Suppose there is a diagram of spaces and maps 
\[
    \begin{tikzcd}[row sep=1.5em, column sep = 1.5em]
    E \arrow[rr] \arrow[dr] \arrow[dd] &&
    F \arrow[dd] \arrow[dr] \\
    & G \arrow[rr] \arrow[dd]&&
    H \arrow[dd] \\ 
    A \arrow[rr,] \arrow[dr] && B \arrow[dr] \\
    & C \arrow[rr] && D
    \end{tikzcd} 
\]
where the bottom face is a homotopy pushout and the four sides are obtained by pulling back with $H \to D$. Then the top face is a homotopy pushout.  \qed
\end{lem}
A common construction of such a cube is to start with a homotopy pushout, and map all four corners of this pushout into a space $Z$, and take the top face to be a diagram of the homotopy fibres of these maps. This is equivalent to the statement that all four sides are homotopy pullbacks, and so the diagram of homotopy fibres is a homotopy pushout. 
 
\section{Polyhedral join products}

In this section, we introduce the polyhedral join product, and prove many fundamental combinatorial results. Recall that the polyhedral join product, as defined by \cite{Abramyan_2019} is 
\[(\underline{K},\underline{L})^{*M} = \bigcup_{\sigma \in M} (\underline{K},\underline{L})^{*\sigma}\] where each $(\underline{K},\underline{L})^{*\sigma}$ is defined as \[(\underline{K},\underline{L})^{*\sigma} = \overset{m}{\underset{i=1} \ast} \; Y_i,\;\; Y_i = \begin{cases} K_i & i \in \sigma \\ L_i & \mathrm{otherwise}. \end{cases}\]
 We now consider some examples of polyhedral join products on pairs of simplicial complexes. 

\begin{ex} \label{pjpex}
Let $K$ be a simplicial complex on the vertex set $[m]$, and let $K_1,\cdots,K_m$ be simplicial complexes on the vertex sets $[k_1],\cdots,[k_m]$ respectively. 
\begin{enumerate}
    \item The \emph{substitution complex} $K(K_1,\cdots,K_m)$ as introduced in \cite{Abramyan_2019} is the polyhedral join product on the pairs $(K_i,\emptyset)$.
    \item  The \emph{composition complex}  $K \langle K_1 ,\cdots, K_m \rangle$ as introduced in \cite{construct} is the polyhedral join product on the pairs $\{(\Delta^{m_i-1},K_i)\}_{i=1}^m$. 
\end{enumerate}
\end{ex}

We now prove some combinatorial properties of the polyhedral join product.

\begin{lem}\label{pjpsubcomplex}
     Let $M$ be a simplicial complex on the vertex set $[m]$, and let $N$ be a subcomplex on the vertex set $[l]$ where $l \leq m$. Then $\kl^{*N}$ is a subcomplex of $\kl^{*M}$. Furthermore, if $N$ is a full subcomplex of $M$, $\kl^{*N}$ is a full subcomplex of $\kl^{*M}$.
\end{lem}
\begin{proof}
    Let $\tau_i$ denote a simplex in $K_i$ and let $\omega_i$ denote a simplex in $L_i$. Every simplex $\sigma \in N$ will generate simplices of the form $(\cup_{i \in \sigma} \tau_i) \cup (\cup_{i \notin \sigma} \omega_i)$. As $N$ is a subcomplex of $M$, $\sigma \in M$, and so $(\cup_{i \in \sigma} \tau_i) \cup (\cup_{i \notin \sigma} \omega_i) \in \kl^{*M}$ for all $\sigma \in N$. Therefore $\kl^{*N}$ is a subcomplex of $\kl^{*M}$. Now let $N$ be a full subcomplex of $M$. Assume $\kl^{*N}$ is not a full subcomplex of $M$, that is, there exists a simplex $\sigma'$ on the vertex set of $\kl^{*N}$ such that $\sigma' \in \kl^{*M}$, but $\sigma' \notin \kl^{*N}$. By definition, $\sigma' = (\cup_{i \in \sigma} \tau_i) \cup (\cup_{i \notin \sigma} \omega_i)$ for some $\sigma \in M$. As $N$ is a full subcomplex of $M$, $\sigma \in N$. By the definition of the polyhedral join product, $(\cup_{i \in \sigma} \tau_i) \cup (\cup_{i \notin \sigma} \omega_i)$ will form a simplex in $\kl^{*N}$, contradicting the claim that $\sigma' \notin \kl^{*N}$. Therefore $\kl^{*N}$ is a full subcomplex of $\kl^{*M}$.
\end{proof}

\begin{lem}\label{ppjpush}
    Let 
    \[\begin{tikzcd}
        N \arrow[r] \arrow[d] 
        &M_B \arrow[d] \\
        M_A \arrow[r]
        & M 
    \end{tikzcd}\]
    be a pushout of simplicial complexes. Then
    \[\begin{tikzcd}
        \kl^{*\overline{N}} \arrow[r] \arrow[d] 
        &\kl^{*\overline{M_B}} \arrow[d] \\
        \kl^{*\overline{M_A}} \arrow[r]
        & \kl^{*M} 
    \end{tikzcd}\] is a pushout of simplicial complexes, where $M$ is on the vertex set $[m]$ , and $\overline{N}$, $\overline{M_A}$ and $\overline{M_B}$ refer to $N$,$M_A$,$M_B$ regarded as simplicial complexes on the vertex set of $M$. 
\end{lem}
\begin{proof}
    As $M$ is finite, the faces in $M$ can be sorted into three categories: (X) the faces in $N$, (Y) the faces in $M_A$ that are not in $N$, and (Z) the faces in $M_B$ that are not in $N$. Therefore we have that
    \[N = \bigcup_{\sigma \in X} \sigma,\]
    \[M_A = \big(\bigcup_{\sigma \in X} \sigma \big) \cup \big(\bigcup_{\sigma' \in Y} \sigma'\big),\]
    \[M_B = \big(\bigcup_{\sigma \in X} \sigma \big) \cup \big(\bigcup_{\sigma'' \in Z} \sigma''\big),\]
    \[M = \big(\bigcup_{\sigma \in X} \sigma \big) \cup \big(\bigcup_{\sigma' \in Y} \sigma'\big) \cup \big(\bigcup_{\sigma'' \in Z} \sigma''\big).\]
    By the definition of the polyhedral join product, we have that $\kl^{*M} = \bigcup_{\sigma \in M} \kl^{*\sigma}$. Therefore 
    \[\kl^{*\overline{N}} = \bigcup_{\sigma \in X} \kl^{*\sigma},\]
    \[\kl^{*\overline{M_A}} = \big(\bigcup_{\sigma \in X} \kl^{*\sigma} \big) \cup \big(\bigcup_{\sigma' \in Y} \kl^{*\sigma'}\big),\]
    \[ \kl^{*\overline{M_B}}= \big(\bigcup_{\sigma \in X} \kl^{*\sigma} \big) \cup \big(\bigcup_{\sigma'' \in Z} \kl^{*\sigma''} \big),\]
    \[\kl^{*M} = \big(\bigcup_{\sigma \in X} \kl^{*\sigma} \big) \cup \big(\bigcup_{\sigma' \in Y} \kl^{*\sigma'}\big) \cup \big(\bigcup_{\sigma'' \in Z} \kl^{*\sigma''}\big).\]

    As $\kl^{*\overline{M_A}} \cup \kl^{*\overline{M_B}} = \kl^{*M}$ and $\kl^{*\overline{M_A}} \cap \kl^{*\overline{M_B}} = \kl^{*\overline{N}}$, the result holds. 
\end{proof}

\begin{lem}\label{ppjjoin}
    There is an equality of simplicial complexes \[\kl^{*(M*N)} = \kl^{*M}*\kl^{*N}.\]
\end{lem}
\begin{proof}
    Simplices in $M*N$ will be of the form $\sigma \cup \tau$, where $\sigma \in M$, $ \tau \in N$.
     Consider $\kl^{*\sigma \cup \tau}$. By definition, this is $*^{m}_{i=1} Y_i$ where $Y_i = K_i$ if $i \in \sigma \cup \tau$ and $L_i$ otherwise. As $\sigma$ and $\tau$ are on disjoint vertex sets, $*^{m}_{i=1} Y_i$ can be written as $(*_{j=1}^{l} Y_i) * (*_{k=l+1}^m Y_k)$ where $Y_j = K_j$ if $j \in \sigma$ and $L_j$ otherwise, and $Y_k=K_k$ if $k \in \tau$, and $L_k$ otherwise. Therefore, $\kl^{*\sigma \cup \tau} = \kl^{*\sigma} * \kl^{*\tau}$.
    Now\begin{multline*}\kl^{*(M*N)}= \bigcup_{\sigma \in M, \tau \in N} \kl^{*\sigma \cup \tau} = \bigcup_{\sigma \in M, \tau \in N} \kl^{*\sigma} * \kl^{*\tau} = \\ \big( \bigcup_{\sigma \in M} \kl^{*\sigma} \big) * \big( \bigcup_{\tau \in N} \kl^{*\tau} \big) = \kl^{*M} * \kl^{*N}. \end{multline*}
\end{proof}

\begin{lem}\label{3.5}
    Let $N$ be a simplicial complex on the vertex set $[m-1]$, and let $\overline{N}$ be the simplicial complex $N$ considered as a simplicial complex on $[m]$. Then $\kl^{*\overline{N}} = \kl^{*N} * L_m$.
\end{lem}
\begin{proof}
    Let $\sigma$ denote a simplex in $N$, and let $\overline{\sigma}$ denote the same simplex viewed as a simplex in $\overline{N}$. Consider $\kl^{*\overline{\sigma}}$. By definition, this is $\overset{m}{\underset{i=1}{\ast}} \; Y_i = (\underset{i \in \overline{\sigma}}{\ast} K_i) \ast (\underset{i \notin \overline{\sigma}}{\ast} L_i)$. As $m$ is a ghost vertex, $m \notin \overline{\sigma}$ for all $\overline{\sigma} \in \overline{N}$, and so 
    \[\kl^{*\overline{\sigma}} = (\underset{i \in \sigma}{\ast} K_i )* (\underset{i \notin \sigma, i \neq m}{\ast} L_i) * L_m = \kl^{*\sigma} * L_m\] where $\sigma$ is now considered a simplex of $N$. Taking the union over all faces $\sigma \in N$ gives the required result. 
\end{proof}
\begin{corollary}\label{pjppushout}
    There is a pushout 
    \[\begin{tikzcd}
        \kl^{*\mathrm{lk}_M(i)} * L_i \arrow[r] \arrow[d]
        & \kl^{*\mathrm{lk}_M(i)} *K_i \arrow[d] \\
        \kl^{*M \setminus i} * L_i \arrow[r]
        & \kl^{*M}.
    \end{tikzcd}\]
    where $\mathrm{lk}_M(i)$ is regarded as a simplicial complex on the vertex set of $M \setminus i$.
\end{corollary}

\begin{proof}
    This follows directly from Lemma \ref{3.5}, Lemma \ref{ppjpush} and the pushout \ref{1}.     
\end{proof}

We can construct the polyhedral join product sequentially as follows. Define $M(i)$ as the polyhedral join product $(\underline{P^i},\underline{Q^i})^{*M}$ where the sequence $(\underline{P^i},\underline{Q^i})$ is defined as \[(P^i_j,Q^i_j) = \begin{cases} (K_j,L_j) & j \leq i\\ (\{j\},\emptyset) & j > i. \end{cases}\]

The vertex set of $M(i)$ shall be regarded as \[\{\{j^1_{1},\cdots,j^1_{k_1}\},\{j^2_{1},\cdots,j^2_{k_2}\},\cdots,\{j^i_{1},\cdots,j^i_{k_i}\},i+1,\cdots,m\}=\{[k_1],\cdots,[k_i],i+1,\cdots,m\}\] This allows us to keep track of the original vertex set of $M$ and the new vertices added by the polyhedral join operation. 

Note that $M(m) = \kl^{*M}$, $M(0) = M$, and each $M(i)$ is a subcomplex of $M(i+1)$. 

\begin{lem}\label{deleteeq}
     Let $\kl$ be a sequence of pairs of simplicial complexes, and let $(\underline{P},\underline{Q})$ be a sequence of simplicial complexes such that 
    \[(P_j,Q_j) = \begin{cases} (K_j,L_j) & j \neq i \\ (\{i\},\emptyset) & j = i. \end{cases}\] Then there is a identity of simplical complexes 
    \[\kl^{*M\setminus i} = ((\underline{P},\underline{Q})^{*M}) \setminus i.\] Furthermore, there is a identity of simplical complexes 
    \[\kl^{*\mathrm{lk}_M(i)} = \mathrm{lk}_{(\underline{P},\underline{Q})^{*M}} (i)\] where $\mathrm{lk}_M(i)$ is considered on the vertex set of $M \setminus i$ and $\mathrm{lk}_{(\underline{P},\underline{Q})^{*M}} (i)$ is considered on the vertex set of $((\underline{P},\underline{Q})^{*M}) \setminus i$. 
\end{lem}
\begin{proof}
    Let $\tau_j$ denote a simplex in $K_j$ and let $\omega_j$ denote a simplex in $L_j$. The simplices of $\kl^{*M\setminus i}$ are all of the form $( \cup_{j \in \sigma} \tau_j ) \cup ( \cup_{j \notin \sigma, j \neq i} \omega_j)$ for $\sigma \in M \setminus i$. As every simplex of $M \setminus i$ is a simplex of $M$, all such simplices will be contained in $((\underline{P},\underline{Q})^{*M})$. As $ i \notin \sigma$, these simplices will be unchanged under the deletion of $i$ from $((\underline{P},\underline{Q})^{*M})$. Therefore $\kl^{*M\setminus i} \subseteq ((\underline{P},\underline{Q})^{*M}) \setminus i$. The simplices of $((\underline{P},\underline{Q})^{*M}) \setminus i$ will be of two forms: those induced by $\gamma \in M$, such that $i \notin \gamma$, and those induced by $\kappa \in M$ such that $i \in \kappa$. The simplices of $((\underline{P},\underline{Q})^{*M}) \setminus i$ induced by $\gamma$ are clearly in $\kl^{*M\setminus i}$. If $\kappa \in M$ contains $i$, then $\kappa$ generates simplices of the form $(( \cup_{j \in \kappa} \tau_j ) \cup ( \cup_{ j\notin \kappa} \omega_j)) \setminus i =  (( \cup_{j \in \kappa,j \neq i} \tau_j ) \cup ( \cup_{j \notin \kappa, j \neq i} \omega_j) \cup \{i\}) \setminus \{i\} = ( \cup_{j\ \in \kappa,j \neq i} \tau_j ) \cup ( \cup_{ j \notin \kappa} \omega_j)$.  
    As $\kappa \in M$ and $i \in \kappa$, we have $\kappa \setminus i \in M \setminus i$. Therefore, $\kl^{*M \setminus i}$ contains all simplices of the form $( \cup_{j \in \kappa, j\neq i} \tau_j ) \cup ( \cup_{ j \notin \kappa} \omega_j)$. As all simplices of $((\underline{P},\underline{Q})^{*M}) \setminus i$ are contained in $\kl^{*M \setminus i}$, we obtain $((\underline{P},\underline{Q})^{*M})\setminus i \subseteq \kl^{*M \setminus i}$. We therefore have $\kl^{*M \setminus i} = ((\underline{P},\underline{Q})^{*M}) \setminus i$. 
   
    Now let us consider the second statement. Simplices of $\kl^{*\mathrm{lk}_M(i)}$ are of the form $(\cup_{j \in \sigma} \tau_j) \cup (\cup_{j \notin \sigma} \omega_j)$ where $\omega \cup i \in M$. By Lemma \ref{pjpsubcomplex}, $\kl^{*\mathrm{lk}_M(i)}$ is a subcomplex of $\kl^{*M\setminus i}$, and so by the above equality, $\kl^{*\mathrm{lk}_M(i)}$ is a subcomplex of $(\underline{P},\underline{Q})^{*M} \setminus i$, and so $(\cup_{j \in \sigma} \tau_j) \cup (\cup_{j \notin \sigma} \omega_j)$ is contained in $(\underline{P},\underline{Q})^{*M} \setminus i$. As $\sigma \in \mathrm{lk}_M(i)$, $\sigma \cup i \in M$, and therefore $(\cup_{j \in \sigma }\tau_j) \cup (\cup_{j\notin \sigma, j \neq i} \omega_j) \cup i$ is a simplex of $(\underline{P},\underline{Q})^{*M}$. By the definition of the link of a vertex, the simplex  $(\cup_{j \in \sigma} \tau_j) \cup (\cup_{j \notin \sigma, j \neq i} \omega_j)$ is contained in $\mathrm{lk}_{(\underline{P},\underline{Q})^{*M})}(i)$. Therefore $\kl^{*\mathrm{lk}_M(i)} \subseteq \mathrm{lk}_{(\underline{P},\underline{Q})^{*M})}(i)$. Now consider a simplex in $\mathrm{lk}_{(\underline{P},\underline{Q})^{*M}}(i)$. Such simplices have the form $(\cup_{j \in \gamma} \tau_j) \cup (\cup_{j \notin \gamma, j \neq i} \omega_j)$ where $(\cup_{j \in \gamma} \tau_j) \cup (\cup_{j \notin \gamma, j \neq i} \omega_j) \cup i \in (\underline{P},\underline{Q})^{*M}$. As $(\cup_{j \in \gamma} \tau_j) \cup (\cup_{j \notin \gamma, j \neq i} \omega_j) \cup \{i\}$ is a simplex of $(\underline{P},\underline{Q})^{*M}$, this implies that $\gamma \cup \{i\}$ is a simplex of $M$. Therefore for all $\gamma$ such that $(\cup_{j \in \gamma} \tau_j) \cup (\cup_{j \notin \gamma, j \neq i} \omega_j) \in \mathrm{lk}_{(\underline{P},\underline{Q})^{*M}}(i)$, the simplex $\gamma \in \mathrm{lk}_M(i)$. Therefore $(\cup_{j \in \gamma} \tau_j) \cup (\cup_{j \notin \gamma, j \neq m} \omega_j) \in \kl^{*\mathrm{lk}_M(j)}$, and so $\mathrm{lk}_{(\underline{P},\underline{Q})^{*M}}(i) \subseteq \kl^{*\mathrm{lk}_M(i)}$. We therefore have the equality $\kl^{*\mathrm{lk}_M(i)} = \mathrm{lk}_{(\underline{P},\underline{Q})^{*M}}(i)$. 
\end{proof}

\begin{lem}\label{eq}  
Let $\kl$ be a sequence of pairs of simplicial complexes, and let $(\underline{P},\underline{Q})$ be a sequence of simplicial complexes such that 
    \[(P_j,Q_j) = \begin{cases} (K_j,L_j) & j \neq i \\ (\{i\},\emptyset) & j = i. \end{cases}\] Let $\overline{(\underline{P},\underline{Q})^{*M}}$ be the simplicial complex defined by the pushout   \[\begin{tikzcd}       \mathrm{lk}_{(\underline{P},\underline{Q})^{*M}}(i) * L_i \arrow[d] \arrow[r]       &  \mathrm{lk}_{(\underline{P},\underline{Q})^{*M}}(i) * K_i \arrow[d] \\      (\underline{P},\underline{Q})^{*M} \setminus i * L_i \arrow[r]   & \overline{(\underline{P},\underline{Q})^{*M}}. \end{tikzcd}\] Then there is an identity of simplicial complexes  \[\kl^{*M} = \overline{(\underline{P},\underline{Q})^{*M}}\] \end{lem}

\begin{proof} 
By Lemma \ref{deleteeq} we can rewrite $(\underline{P},\underline{Q})^{*M} \setminus i$ and $\mathrm{lk}_{(\underline{P},\underline{Q})^{*M}}(i)$. Thefore the pushout defining $ \overline{(\underline{P},\underline{Q})^{*M}}$ can be written as    \[\begin{tikzcd}       \kl^{*\mathrm{lk}_M(i)}* L_i \arrow[d] \arrow[r]       &  \kl^{*\mathrm{lk}_M(i)} * K_i\arrow[d] \\        \kl^{*M \setminus i} * L_i \arrow[r]       & \overline{(\underline{P},\underline{Q})^{*M}}.   \end{tikzcd}\]   By Lemma \ref{pjppushout}, $\overline{(\underline{P},\underline{Q})^{*M}} = \kl^{*M}$. \end{proof}

Recall that $M(i)$ is defined as polyhedral join product $(\underline{P^i},\underline{Q^i})^{*M}$ where the sequence $(\underline{P^i},\underline{Q^i})$ is defined as \[(P^i_j,Q^i_j) = \begin{cases} (K_j,L_j) & j \leq i\\ (\{j\},\emptyset) & j > i. \end{cases}\]By Lemma \ref{eq}, with $\kl = ((\underline{P^{i+1}},\underline{Q^{i+1}})$ and $(\underline{P},\underline{Q})= (\underline{P^i},\underline{Q^i})$, we obtain $M(i+1)$ from $M(i)$ via the  pushout 
\[\begin{tikzcd}
    \mathrm{lk}_{M(i)}(i+1) * L_{i+1} \arrow[r] \arrow[d] 
    & \mathrm{lk}_{M(i)}(i+1) * K_{i+1} \arrow[d] \\
    M(i) \setminus \{i+1\} * L_{i+1} \arrow[r]
    & M(i+1).
\end{tikzcd}\]

We therefore obtain the following result. 

\begin{lem}\label{seq}
    The polyhedral join product can be constructed sequentially. That is, given a simplicial complex $M$ on the vertex set $[m]$, we can obtain $\kl^{*M}$ via the following sequence 
    \[M \subset M(1) \subset M^2 \subset \cdots \subset M(m) = \kl^{*M}\]
    where at each point, $M(i+1)$ is obtained from $M(i)$ via the pushout 
    \[\begin{tikzcd}
    \mathrm{lk}_{M(i)}(i+1) * L_{i+1} \arrow[r] \arrow[d] 
    & \mathrm{lk}_{M(i)}(i+1) * K_{i+1} \arrow[d] \\
    M(i) \setminus \{i+1\} * L_{i+1} \arrow[r]
    & M(i+1).
    \end{tikzcd} \eqno  \]\qed
\end{lem}

Specialising to the substitution complex, we give specific results concerning full subcomplexes. 

\begin{lem}\label{sub1}
    Let $K$ be a simplicial complex on $[m]$ vertices, and let $L$ be a full subcomplex of $K$ on $[l]$ vertices where $l \leq m$. Then $L(K_1,\cdots,K_l)$ is a full subcomplex of $K(K_1,\cdots,K_m)$.
\end{lem}
\begin{proof}
This follows from Lemma \ref{pjpsubcomplex}.
\end{proof}

\begin{corollary}\label{full1}
     Let $K(K_1,\cdots,K_n)$ be a substitution complex. Then each $K_i$ is a full subcomplex on its original vertex set. 
\end{corollary}
\begin{proof}
    As each vertex is a full subcomplex, by Lemma \ref{sub1} $\{i\}(K_1,\cdots,K_m) = K_i$ is a full subcomplex of $K(K_1,\cdots,K_m)$. 
\end{proof}
\begin{lem}\label{full2}
   Let $K(K_1,\cdots,K_n)$ be a substitution complex. Then $K(K_1,\cdots,K_n)$ contains a full subcomplex $K'$, such that there is a simplicial equivalence $K'=K$.  
\end{lem}
\begin{proof}
    Let $v_i$ denote a vertex belonging to $K_i$. Consider the full subcomplex on the vertex set $\{v_1,\cdots,v_n\} \in K(K_1,\cdots,K_n)$. We will show that this is isomorphic to a copy of $K$. Let $\sigma = (j_1,\cdots,j_k)$ be a simplex of $K$. By definition of the substitution operation, $\sigma_{j_1} \cup \cdots \cup \sigma_{j_k} \in K(K_1,\cdots,K_n)$ for all $\sigma_{j_i} \in K_{j_i}$. Letting each $\sigma_{j_i} = v_{j_i} \in K_{j_i}$, we have that $(v_{j_1},\cdots,v_{j_k}) \in K(K_1,\cdots,K_n)$. Therefore, for every simplex $\sigma \in K$, we have a copy $\sigma' = (v_{j_1},\cdots,v_{j_k}) \in K(K_1,\cdots,K_n)$, and so a copy of $K$ is contained in $ K(K_1,\cdots,K_n)$. Denote this copy $K'$. We now show $K'$ is a full subcomplex of $K(K_1,\cdots,K_n)$. Assume that $K'$ is not a full subcomplex, so there is a face $\tau'$ on some subset of the vertex set $\{v_1,\cdots,v_n\}$  that is is not a face of $K'$ but is a face of $K(K_1,\cdots,K_n)$. Let $\tau' = (v_{j_1},\cdots,v_{j_k}) \in K(K_1,\cdots,K_n)$. By definition of substitution, $\tau = (j_1,\cdots,j_k)$ is a simplex in $K$. But then as there is a copy of every simplex in $K$ replicated on the vertex set $\{v_1,\cdots,v_n\}$, we have that $\tau'$ must be a face on the vertex set $\{v_1,\cdots,v_n\}$, and so $K'$ is a full subcomplex of $K(K_1,\cdots,K_n)$. 
\end{proof}
In general, the polyhedral join product does not contain a copy of the original simplex $M$ as a full subcomplex. This property is unique to the substitution complex. 

\section{A generalised homotopy equivalence for loop spaces of certain polyhedral products}

In this section, we generalise results from \cite{soton467338}. 
Let $M$ be a simplicial complex on $[m]$, and let $N$ be a subcomplex of $M$. Let $K$ be a simplicial complex on $[k]$, and let $L$ be a subcomplex of $K$. Consider $N$ and $L$ to be on the vertex sets of $M$ and $K$ respectively. Define $Q$ to be the pushout
     \begin{equation}\label{pushout}
         \begin{tikzcd}
         N*L \arrow[r] \arrow[d]
         & N*K \arrow[d] \\
         M*L \arrow[r]
         & Q. 
     \end{tikzcd}\end{equation}
The simplicial complex $Q$ is considered to be on the vertex set $[m+k]$. Applying Lemma \ref{join} and Lemma \ref{funda2}  gives the following result
 \begin{lem}
    There exists a pushout 
    \[\begin{tikzcd}
         \uxa^N \times \uxa^L \arrow[r] \arrow[d]
         & \uxa^N \times \uxa^K \arrow[d] \\
         \uxa^M \times \uxa^L \arrow[r]
         & \uxa^{Q}. 
     \end{tikzcd} \eqno  \] \qed
 \end{lem}

There are inclusion maps from $M*L$ and $N*K$ into $M*K$, and the inclusions $M*K$ coincide on $N*L$. Therefore there is a pushout map $Q \to M*K$. This induces the following pushout map on topological spaces:
\begin{equation}\label{baseofcube}
    \begin{tikzcd}
	{\uxa^N \times \uxa^L} & {\uxa^N \times \uxa^K} \\
	{\uxa^M \times \uxa^L} & {\uxa^{Q}} \\
	&& {\uxa^M \times \uxa^K. }
	\arrow[from=1-1, to=1-2]
	\arrow[from=1-1, to=2-1]
	\arrow[from=2-1, to=2-2]
	\arrow[from=1-2, to=2-2]
	\arrow[curve={height=-18pt}, from=1-2, to=3-3]
	\arrow[curve={height=18pt}, from=2-1, to=3-3]
	\arrow[dashed, from=2-2, to=3-3,"\Theta"]
\end{tikzcd}\end{equation}

The maps $N * K \to M * K$ and $M*L \to M*K$ are both the join of an inclusion and the identity, and so on the level of polyhedral products, they induce the maps\[\uxa^{N} \times \uxa^{K} \xrightarrow{g \times 1} \uxa^{M} \times \uxa^K\]and \[\uxa^{M} \times \uxa^L \xrightarrow{1 \times h} \uxa^{M} \times \uxa^L.\] If all four corners of the pushout (\ref{baseofcube}) are included into $\uxa^{M} \times \uxa^K$, we obtain homotopy fibrations 
\[F \to \uxa^{Q} \xrightarrow{\Theta} \uxa^{M} \times  \uxa^K\]
\[G \to \uxa^{N} \times \uxa^K  \xrightarrow{g\times 1} \uxa^{M} \times \uxa^K\]
\[H \to \uxa^{M} \times \uxa^L \xrightarrow{1 \times h} \uxa^{M} \times \uxa^K\]
\[G \times H \to \uxa^{N} \times \uxa^L \xrightarrow{h \times i} \uxa^{M} \times \uxa^K\]
 where $F$ is the homotopy fibre of $\Theta$, $G$ is the homotopy fibre of the map $g$, and $H$ is the homotopy fibre of $h$. As the homotopy fibres are the result of mapping each corner of the pushout into a common base, we obtain the homotopy commutative cube 
 \begin{equation}
    \begin{tikzcd}[row sep=1.5em, column sep = 1.5em]
    G \times H \arrow[rr] \arrow[dr] \arrow[dd] &&
    G \arrow[dd] \arrow[dr] \\
    & H \arrow[rr] \arrow[dd]&&
    F \arrow[dd] \\ 
    \uxa^{N} \times  \uxa^L \arrow[rr,] \arrow[dr] && \uxa^{N} \times \uxa^{K} \arrow[dr] \\
    & \uxa^{M} \times \uxa^L \arrow[rr] && \uxa^{Q}
    \end{tikzcd}
 \end{equation}

in which all four sides are homotopy pullbacks. Since the bottom face is a homotopy pushout, by Lemma \ref{cube}, the top face is a homotopy pushout. 

\begin{lem}\label{jfib}
    The maps $G \times H \to G$ and $G \times H \to H$ in (2) can be chosen to be projections, implying that there is a homotopy equivalence $F \simeq G*H$. 
\end{lem}
\begin{proof}
    Consider the following diagram
\[\begin{tikzcd}
	{G \times H} & {\uxa^{N} \times \uxa^L } & {\uxa^{M} \times \uxa^K} \\
	{* \times H} & {\uxa^{M} \times \uxa^L} & {\uxa^{M}\times \uxa^K.}
	\arrow[from=1-1, to=1-2]
	\arrow["{g \times h}", from=1-2, to=1-3]
	\arrow["{* \times 1}",from=1-1, to=2-1]
	\arrow[from=2-1, to=2-2]
	\arrow["{1 \times h}", from=2-2, to=2-3]
	\arrow[Rightarrow, no head, from=1-3, to=2-3]
	\arrow["{g \times 1}", from=1-2, to=2-2]
\end{tikzcd}\]
As the fibrations are taken over the common  base $\uxa^M \times \uxa^K$, the left square is the homotopy pullback appearing in the rear face of the cube. This diagram is the product of the fibration diagram for the factors on the left and the fibration diagram for the factors on the right. Observe $* \times 1$ is the projection of the righthand factor. Thus the map $G \times H \to G$ is homotopic to the projection. The argument for the map $G \times H \to G$ being a projection is similar. It is well known that the homotopy pushout of projections is a join, and so the homotopy pushout in the top face of the cube implies that $F \simeq G*H$. 
\end{proof}

\begin{thm}\label{the1}
    Let $Q$ be the pushout \ref{pushout}. There is a homotopy fibration \[G*H\to \uxa^{Q} \to \uxa^{M} \times \uxa^K\] where $G$ is the homotopy fibre of $\uxa^N \to \uxa^{M}$ and $H$ is the homotopy fibre of $\uxa^L \to \uxa^K$. This fibration splits after looping, giving a homotopy equivalence \[\Omega\uxa^{Q} \simeq \Omega\uxa^{M} \times \Omega \uxa^{K} \times \Omega(G*H).\]
\end{thm}
\begin{proof}
    If $F$ is the homotopy fibre of the map $\Theta: \uxa^{Q} \to \uxa^M \times \uxa^K$, then by Lemma~\ref{jfib}, $F \simeq G*H$. This establishes the asserted homotopy fibration. For the splitting, consider the composite $\uxa^{M} \times\uxa^{L} \to \uxa^{Q} \to \uxa^{M} \times \uxa^{K}$ which is $1 \times h$. Restricting to $\uxa^{M}$ gives inclusion into the left factor. Similarly, restricting the composite $\uxa^N \times \uxa^K \to \uxa^{Q} \to \uxa^M \times \uxa^K$ to $\uxa^K$ gives inclusion into the right factor. Taking the wedge sum of these therefore gives the composite \[\uxa^{M} \vee \uxa^K \to \uxa^{Q} \to \uxa^{M} \times \uxa^K\] which is the inclusion of the wedge into the product. The inclusion of the wedge into the product has a right homotopy inverse after looping, so the map $\Omega\uxa^{Q} \to \Omega\uxa^{M} \times \Omega\uxa^K$ has a right homotopy inverse, implying that the fibration splits after looping. 
\end{proof}
\section{Loop spaces of polyhedral products associated with polyhedral join products}

In this section, we apply Theorem \ref{the1} to the polyhedral join product. By Corollary \ref{pjppushout} there exists a pushout

 \[\begin{tikzcd}
        \kl^{*\mathrm{lk}_M(m)} * L_m \arrow[r] \arrow[d]
        & \kl^{*\mathrm{lk}_M(m)} *K_m \arrow[d] \\
        \kl^{*M \setminus m} * L_m \arrow[r]
        & \kl^{*M}.
    \end{tikzcd}\]

This is a special case of the pushout in \ref{pushout}, where $N = \kl^{*\mathrm{lk}_M(m)}$, $L = L_m$, $M = \kl^{*M \setminus m}$ and $K = K_m$. Applying Theorem \ref{the1}, we obtain a homotopy equivalence 

\[\Omega \uxa^{\kl^{*M}} \simeq \Omega \uxa^{\kl^{*M \setminus m}} \times \Omega \uxa^{K_m} \times \Omega(G_m*H_m),\] where $G_m$ is the homotopy fibre of the map $\uxa^{\kl^{*\mathrm{lk}_M(m)}} \to \uxa^{\kl^{*M \setminus m}}$ and $H_m$ is the homotopy fibre of the map $\uxa^{L_m} \to \uxa^{K_m}$. By iterating this homotopy equivalence with respect to the vertices of $M$, we prove Theorem \ref{mainthrm1}.

\begin{proof}[Proof of Theorem \ref{mainthrm1}]
    By Theorem \ref{the1}, there is a homotopy equivalence \[\Omega \uxa^{\kl^{*M}} \simeq \Omega \uxa^{\kl^{*M \setminus m}} \times \Omega \uxa^{K_m} \times \Omega(G_m*H_m).\] Now consider $\kl^{*M \setminus m}$. If $M \setminus m$ is a simplex, then by Example \ref{pjpex}, $\kl^{*M \setminus m} = K_1*\cdots*K_{m-1}$, and so $\uxa^{\kl^{*M \setminus m}} \simeq \uxa^{K_1} \times \cdots \times \uxa^{K_m}$, and the result follows. If not, the simplicial complex $\kl^{*M \setminus m}$ is a polyhedral join product, and therefore can be constructed as in \ref{pushout}, with $N = \kl^{*\mathrm{lk}_{*M\setminus m}}$, $L = L_{m-1}$, $M = \kl^{* M \setminus \{m,m-1\}}$ and $K = K_{m-1}$. We can apply Theorem \ref{the1} to $\uxa^{\kl^{*M \setminus m}}$ to obtain a homotopy equivalence \[\Omega \uxa^{\kl^{*M \setminus m}} \simeq \Omega \uxa^{\kl^{*M \setminus \{m,m-1\}}} \times \Omega \uxa^{K_{m-1}} \times \Omega(G_{m-1}*H_{m-1}).\] Substituting this into our previous expression gives a homotopy equivalence
     \[\Omega \uxa^{\kl^{*M}} \simeq \Omega \uxa^{\kl^{*M \setminus \{m,m-1\}}} \times \prod_{i=m-1}^m \Omega \uxa^{K_{i}} \times \prod_{j=m+1}^m \Omega(G_{j}*H_{j}).\]
     We continue this process on $\uxa^{\kl^{*M \setminus \{m,m-1\}}}$ until we reach the vertex $k$ such that $M \setminus \{m,\cdots,k\}$ is a simplex.
\end{proof}

\section{Simplicial operations preserving the property of being in \cal{P}}

Recall that $\cal{P}$ is the full subcategory of topological spaces that are homotopy equivalent to a product of spheres and loops on spheres. Restricting to the case where $\uxa = \caa$, let $M$ be a simplicial complex such that $\Omega \caa^M \in \cal{P}$. The polyhedral join product $\kl^{*M}$ can be considered as an operation on simplicial complexes, and as such, it is natural to ask when such an operation preserves the property of having a loop space on an associated polyhedral product in $\cal{P}$. In this section we give some general conditions that ensure $\Omega\caa^{\kl^{*M}} \in \cal{P}$. We then determine that for the substitution complex and the composition complex, we can guarantee some of these conditions will always be met. 
\begin{thm}\label{pjppreserve}
    Let $M$ be a simplicial complex on the vertex set $[m]$, and let $K_i$ be a simplicial complex on the vertex set $[k_i]$. Let $\caa$ be a sequence of pairs $(CA_j,A_j)$ where $j \in [m -1 + k_i]$. Suppose $\Sigma A_j \in \cal{W}$ for all  $j \in [m -1 + k_i]$. Suppose $\Omega \caa^M \in \cal{P}$, and let $Q$ be the simplicial complex defined by the following pushout 
    \[\begin{tikzcd}
        \mathrm{lk}_M(i) * L_i \arrow[r] \arrow[d]
        &  \mathrm{lk}_M(i) * K_i \arrow[d] \\
        M \setminus i * L_i \arrow[r] 
        & Q. 
    \end{tikzcd}\]
    Let $H_i$ be the homotopy fibre of the map $\caa^{L_i} \to \caa^{K_i}$. If $\Omega \caa^{K_i} \in \cal{P}$ and $\Sigma H_i \in \cal{W}$ then $\Omega \caa^{Q} \in \cal{P}$.
\end{thm}
\begin{proof}
    Applying Lemma \ref{fibsplit1} to $\caa^M$ gives $\Omega \caa^{M} \simeq \Omega \caa^{M \setminus i} \times \Omega (\Sigma G_i \wedge A_i)$. As $\Omega \caa^M \in \cal{P}$, and both $\Omega \caa^{M\setminus i}$ and $\Omega (\Sigma G_i \wedge A_i)$ retract off $\Omega \caa^M$ by Lemma \ref{retractinp} both $\Omega \caa^{M\setminus i}$ and $\Omega (\Sigma G_i \wedge A_i)$ are in $\cal{P}$. By Lemma \ref{1.3}, $\Sigma (G_i \wedge A_i) \in \cal{W}$. Now consider $\caa^{Q}$. Applying Theorem \ref{the1} gives \[\Omega \caa^{Q} \simeq \Omega \caa^{M \setminus i} \times \Omega\caa^{K_i} \times \Omega(\Sigma G_i \wedge H_i).\] As $\Sigma H_i \in \cal{W}$, and $\Sigma H_i$ is at least as connected as $\Sigma A_i$ by Lemma \ref{connected1}, $\Sigma G_i \wedge H_i \in \cal{W}$. Therefore $\Omega\Sigma (G_i \wedge H_i) \in \cal{P}$ by Lemma \ref{loopinp}. Hence each element on the right hand side of the homotopy equivalence is in $\cal{P}$, and therefore $\Omega \caa^{Q} \in \cal{P}$.   
\end{proof}

Recall that by Lemma \ref{seq}, the polyhedral join product can be constructed sequentially via an iterated pushout construction. More specifically, there is a sequence of simplicial complexes \[M \subset M(1) \subset M(2) \subset \cdots \subset M(m) = \kl^{*M}\]
    where at each point, $M(i+1)$ is obtained from $M(i)$ via the pushout 
    \[\begin{tikzcd}
    \mathrm{lk}_{M(i)}(i+1) * L_{i+1} \arrow[r] \arrow[d] 
    & \mathrm{lk}_{M(i)}(i+1) * K_{i+1} \arrow[d] \\
    M(i) \setminus \{i+1\} * L_{i+1} \arrow[r]
    & M(i+1).
    \end{tikzcd}\] This, combined with Theorem \ref{pjppreserve}, allows us to prove Theorem \ref{mainthrm2}. 

\begin{proof}[Proof of Theorem \ref{mainthrm2}]
    We proceed by induction on $M(i)$. The base case is $M(1)$. The simplicial complex $M(1)$ is defined as the pushout
    \[\begin{tikzcd}
        \mathrm{lk}_M(1) * L_1 \arrow[r] \arrow[d]
        &  \mathrm{lk}_M(1) * K_1 \arrow[d] \\
        M \setminus 1 * L_1 \arrow[r] 
        & M(1).
    \end{tikzcd}\] By assumption  $\Omega \caa^{K_1} \in \cal{P}$, and $\Sigma H_1 \in \cal{W}$ and so by Theorem \ref{pjppreserve}, we obtain $\Omega \caa^{M(1)} \in \cal{P}$. Now suppose that for all $j < i$, $\Omega \caa^{M(i)} \in \cal{P}$. Consider $M(i)$. This simplicial complex is defined by the pushout 
    \[\begin{tikzcd}
    \mathrm{lk}_{M(i-1)}(i) * L_{i} \arrow[r] \arrow[d] 
    & \mathrm{lk}_{M(i-1)}(i) * K_{i} \arrow[d] \\
    M(i-1) \setminus i * L_{i} \arrow[r]
    & M(i+1).
    \end{tikzcd}\]
     By induction, $\Omega \caa^{M(i-1)} \in \cal{P}$, and by assumption, $\Omega \caa^{K_i} \in \cal{P}$ and $\Sigma H_i \in \cal{W}$. Therefore we can apply Theorem \ref{pjppreserve}, with $M = M(i-1)$ and $i=m$ to obtain $\Omega \caa^{M(i)} \in \cal{P}$. Therefore $\Omega \caa^{M(i)} \in \cal{P}$ for all $i \in [m]$ by induction. When $i=m$ by definition $M(m) = \kl^{*M}$, and so the result holds. 
\end{proof}
Theorem \ref{mainthrm2} tells us that given $M$ such that $\Omega \caa^{M} \in \cal{P}$, we can construct a new simplicial complex $\kl^{*M}$ such that $\Omega \caa^{\kl^{*M}} \in \cal{P}$, subject to conditions on $K_i$ and the homotopy fibre $H_i$. Guaranteeing $\Sigma H_i \in W$ is not so simple. We can give conditions on the map $\caa^{L_i} \to \caa^{K_i}$ to ensure $\Sigma H_i \in \cal{W}$.   
\begin{lem}\label{nullmap}
     Let $M$ be a simplicial complex on $[m]$ vertices. Let $\kl^{*M}$ be a polyhedral join product and let $K_i$ be a simplicial complex on the vertex set $[k_i]$ for $i \in [m]$. Let $\caa$ be a sequence of pairs $(CA_j,A_j)$ where $1 \leq j \leq k_1 + \cdots + k_m$. Suppose the pairs $\kl$ are such that the map $\caa^{L_i} \to \caa^{K_i}$ is null homotopic. If $\Sigma \caa^{L_i} \in \cal{W}$ and $\Omega \caa^{K_i} \in \cal{P}$ for all $i$, then $\Omega \caa^{\kl^{*M}} \in \cal{P}$. 
\end{lem}
\begin{proof}
    As the map $\caa^{L_i} \to \caa^{K_i}$ is null homotopic, its
    homotopy fibre, denoted $H_i$, will be homotopy equivalent to $\caa^{L_i} \times \Omega \caa^{K_i}$. Therefore \[\Sigma H_i \simeq \Sigma \caa^{L_i} \vee \Sigma \Omega \caa^{K_i} \vee \Sigma (\caa^{L_i} \wedge \Omega \caa^{K_i}).\] As $\Sigma \caa^{L_i} \in \cal{W}$ by assumption, and $\Sigma \Omega \caa^{K_i} \in \cal{W}$ by Lemma \ref{susfinw}, all elements on the right hand side are in $\cal{W}$. Therefore we have $\Sigma H_i \in \cal{W}$. By Theorem \ref{mainthrm2}, then $\Omega \caa^{\kl^{*M}} \in \cal{P}$.
\end{proof}

The conditions laid out in Lemma \ref{nullmap} are generally easy to fulfill. For example, the pairs $(\Delta^{n_i-1}_k,\Delta^{n_i-1}_{k-1})$ satisfy such a condition, as $\caa^{\Delta^{n-1}_k} \in \cal{W}$ for $0 \leq k \leq n-1$ (\cite{IRIYE2013716}), and the map $\caa^{\Delta^{n-1}_k} \to \caa^{\Delta^{n-1}_{k+1}}$ is null homotopic (\cite{porter}). Recall that by \cite{bahri2008polyhedral} there is a homotopy equivalence \[\Sigma\caa^K \simeq \Sigma \big(\bigvee_{I \notin K} |K_I|*\widehat{A}^I \big).\] Therefore, if for all $I \notin K$, $|K_I| \in \cal{W}$ and $\Sigma A_i \in \cal{W}$ for all $i \in [m]$, then $\Sigma\caa^K \in \cal{W}$.

More generally, we can show that the substitution complex satisfies this condition. We first require the following proposition. 

\begin{prop}\cite{grbic2011homotopy}\label{nullhomotopic}
    Let $K$ be a simplicial complex on the vertex set $[m]$, and assume $K$ has no ghost vertices. Then the inclusion $\prod^m_{i=1} A_i \to \caa^K$ is null homotopic, and the homotopy fibre of this map is $\big(\prod^m_{i=1} A_i \big)\times \Omega \caa^{K}$.
\end{prop}

\begin{lem}\label{subpreserve}
     Let the sequence of pairs $\caa$ be as in Theorem \ref{mainthrm2}. Let $K(K_1,\cdots,K_m)$ be a substitution complex. If $\Omega \caa^K,\Omega \caa^{K_1},\cdots$  $\Omega \caa^{K_m} \in \cal{P}$, then $\Omega \caa^{K(K_1,\cdots,K_m)} \in \cal{P}$. 
\end{lem}
\begin{proof}
    Recall that the substitution complex is a special case of the polyhedral join product on the simplicial pairs $(K_i, \emptyset)$. Therefore $H_i$ is the homotopy fibre of the map $\prod_{j \in K_i} A_j \to \caa^{K_i}$. By Lemma \ref{nullhomotopic}, this map is null homotopic. As by assumption, $\Sigma A_j \in \cal{W}$ for all $j$, the conditions of Lemma \ref{nullmap} are satisfied, and $\Omega \caa^{K(K_1,\cdots,K_m)} \in \cal{P}$.
\end{proof}

\begin{lem}
Let the sequence of pairs $\caa$ be as in Theorem \ref{mainthrm2}. Let $K \langle K_1 ,\cdots, K_m \rangle$ be a composition complex. If $\Omega \caa^K \in \cal {P}$ and $\Sigma \caa^{K_i} \in \cal{W}$ for all $i \in [m]$, then $\Omega \caa^{K\langle K_1,\cdots,K_m\rangle} \in \cal{P}$. 
\end{lem}
\begin{proof}
     Recall that the composition complex is a special case of the polyhedral join product on the simplicial pairs $(\Delta^{[k_i-1]}, K_i)$. Therefore $H_i$ is the homotopy fibre of the map $\caa^{K_i} \to \prod_{j \in K_i} CA_j$. As $\prod_{j \in K_i} CA_j$ is contractible, this map is null homotopic, so $H_i \simeq \caa^{K_i}$. By assumption $\Sigma \caa^{K_i} \in \cal{W}$, and so by Lemma \ref{nullmap}, $\Omega \caa^{K(K_1,\cdots,K_m)} \in \cal{P}$.
  \end{proof}

There are other pairs of simplicial complexes $\kl$ for which we can determine if $\Sigma H_i \in \cal{W}$.

\begin{prop}\label{fullsubex}
    Suppose $\kl$ are pairs such that for each $i \in [m]$, the simplicial complex $L_i$ is a full subcomplex of $K_i$. Let $\clxx$ be a sequence of pairs where $\Omega X_i \in \cal{P}$ for all $i \in [m]$. Suppose $\Omega \clxx^{M} \in \cal{P}$, and suppose $\clxx^{K_i} \in \cal{P}$ for all $i \in [m]$. Then $\Omega \clxx^{\kl^{*M}} \in \cal{P}$.
\end{prop}
\begin{proof}
   To apply Theorem \ref{mainthrm2} we need to show $\Sigma H_i$, the homotopy fibre of the map $\clxx^{L_i} \to \clxx^{K_i}$ is an element of $\cal{W}$. First consider $\Omega \ux^M$. By \cite{MR2321037} there is a homotopy fibration 
   \[\clxx^M \to \ux^M \to \prod_{i=1}^m \Omega X_i\] that splits after looping
    \[\Omega \ux^M \simeq \prod_{i=1}^m \Omega X_i \times \Omega \clxx^M.\]
    As every element on the right hand side of this homotopy equivalence is in $\cal{P}$, we obtain $\Omega \ux^M \in \cal{P}$. Via an identical argument, $\Omega \ux^{K_i} \in \cal{P}$ for all $i \in [m]$. Let $F_i$ denote the homotopy fibre of $\ux^{L_i} \to \ux^{K_i}$. As $L_i$ is a full subcomplex of $K_i$, by Lemma \ref{retract1} this map has a left homotopy inverse, and so the connecting map $\delta: \Omega \ux^{K_i} \to F_i$ has a right homotopy inverse. We therefore have a homotopy equivalence $\Omega \ux^{K_i} \simeq \Omega \ux^{L_i} \times F_i$. If $\Omega \ux^{K_i} \in \cal{P}$, by Lemma \ref{retractinp} $F_i$ is also in $\cal{P}$. By Lemma \ref{sw}, $\Sigma F_i$ must therefore be in $\cal{W}$. Now we can determine the homotopy type of $\Sigma H_i$. Consider the following homotopy fibration diagram \[\begin{tikzcd}H_i & {\clxx^{L_i}} & {\clxx^{K_i}} \\F_i & {\ux^{L_i}} & {\ux^{K_i}} \\\ast & {\prod_{j=1}^{m_i}X_j} & {\prod_{j=1}^{m_i}X_j}\arrow["{  }", from=1-1, to=1-2]\arrow[from=1-2, to=1-3]\arrow["\simeq"', from=1-1, to=2-1]\arrow[from=2-1, to=2-2]	\arrow[from=2-2, to=2-3]	\arrow[from=2-1, to=3-1]	\arrow[from=3-1, to=3-2]	\arrow["=",from=3-2, to=3-3]	\arrow[from=2-3, to=3-3]	\arrow[from=1-3, to=2-3]	\arrow[from=1-2, to=2-2]	\arrow[from=2-2, to=3-2]\end{tikzcd}\]   
    where the homotopy fibrations in the middle and right column are exactly as defined in \cite{MR2321037}. From this diagram, there is a homotopy equivalence $F_i \simeq H_i$, and as $\Sigma F_i \in \cal{W}$ we obtain $\Sigma H_i \in \cal{W}$. As all the conditions of Theorem \ref{mainthrm2} have been met, it follows that $\clxx^{\kl^{*M}} \in \cal{P}$. 
\end{proof}
\section{The substitution complex and new examples}
In this section we show that the substitution operation can be used to generate new examples of simplicial complexes such that $\Omega \caa^{K} \in \cal{P}$. Recall that there are two known families of simplicial complexes where $\Omega  \caa^K \in \cal{P}$: those $K$ such that $\caa^K \in \cal{W}$, and those $K$ that are $k$-skeletons of a flag complex. We wish to construct new examples of simplicial complexes where $\Omega \caa^K \in \cal{P}$ but $\caa^K \notin \cal{W}$ and $K$ is not the $k$-skeleton of a flag complex. To do this, we focus on the substitution complex, as we can =determine some full subcomplexes in this case. We first give sufficient conditions to ensure that $\caa^{K(K_1,\cdots,K_m)}$ is not in $\cal{W}$.

\begin{lem}\label{notwedge1}
   Let $K(K_1,\cdots,K_m)$ be a substitution complex where at least one $K_i$ is such that $\caa^{K_i} \notin \cal{W}$ . Then $\caa^{K(K_1,\cdots,K_m)} \notin \cal{W}$.  
\end{lem}
\begin{proof}
    Assume that $\caa^{K(K_1,\cdots,K_m)} \in \cal{W}$. By Corollary \ref{full1}, $K_i$ is a full subcomplex of $K(K_1,\cdots,K_m)$. Therefore by Lemma \ref{retract1} ,$\caa^{K_i}$ retracts off $\caa^{K(K_1,\cdots,K_m)}$. By Lemma~ \ref{retractinw} this implies that $\caa^{K_i}$ has the homotopy type of a wedge of spheres, giving a contradiction. Therefore, $\caa^{K(K_1,\cdots,K_m)}$ cannot have the homotopy type of a wedge of spheres.
\end{proof}

\begin{lem}\label{notwedge2}
    Let $K(K_1,\cdots,K_m)$ be a substitution complex and let $K$ be such that $\caa^K \notin \cal{W}$. Then $\caa^{K(K_1,\cdots,K_m)}\notin \cal{W}$
\end{lem}
\begin{proof}
    Assume that $\caa^{K(K_1,\cdots,K_m)} \in \cal{W}$. By Corollary \ref{full1}, there is a copy of $K$ contained in $K(K_1,\cdots,K_m)$ which is a full subcomplex. Therefore by Lemma \ref{retract1}, we have $\caa^{K}$ retracting off $\caa^{K(K_1,\cdots,K_m)}$.By Lemma \ref{retractinw} this implies that $\caa^{K}$ has the homotopy type of a wedge of spheres, giving a contradiction.  Therefore $\caa^{K(K_1,\cdots,K_m)}$ is not homotopy equivalent to a wedge of spheres, as $\caa^K \notin \cal{W}$. 
\end{proof}

We now wish to show that we can build substitution complexes that are not the $k$-skeletons of flag complexes. We first deal with the case where $K$ is not flag. Recall that a simplicial complex $K$ is \emph{flag} if every set of vertices of $K$ that are pairwise connected by edges form a face of $K$, and the \emph{$k$-skeleton} of a simplicial complex $K$ is the collection of simplices of dimension $k$ or less. A minimal missing face of $K$ is a subset $\omega \subseteq [m]$ such that $\omega \notin K$ but every proper subset of $\omega$ is a simplex in $K$. We denote the set of missing faces of $K$ as $MF(K)$, and the set of minimal missing faces of $K$ as $MMF(K)$. If $K$ is a flag complex, its minimal missing faces are on no more than two vertices. 

\begin{lem}\cite{Abramyan_2019}\label{mf}
    The minimal missing faces of $K(K_1,\cdots,K_m)$ are precisely the following \[MMF(K_1) \sqcup \cdots \sqcup MMF(K_m) \sqcup \bigsqcup_{\Delta(i_1,\cdots,i_k)\in MMF(K)} MMF_{i_1,\cdots,i_k}(K(K_1,\cdots,K_m)) \] where \[MMF_{i_1,\cdots,i_k}(K(K_1,\cdots,K_m)) = 
\{\Delta(j_1,\cdots, j_k) \,|\, j_l \in K_{i_l}
, l = 1,\cdots, k\}. \eqno\qed \] 
\end{lem}

\begin{lem}\label{notflag}
Let $K$ be a simplicial complex on the vertex set $[m]$ that is not flag. Then the complex $K(K_1,\cdots,K_m)$ is not flag.
\end{lem}

\begin{proof}
    Relabeling vertices as required, let $\{1,\cdots,k\}$, $k \geq 2$ be a set of pairwise connected vertices that do not form a face of of $K$. Let $v_i$ denote a vertex in $K_i$, and consider the vertex set $\{v_1,\cdots,v_k\} \in K(K_1,\cdots,K_m)$. This does not form a simplex in $K(K_1,\cdots,K_m)$ as $\{1,\cdots,k\}$ is not a face of $K$. However, as each pair $\{i,j\} \subseteq \{1,\cdots,k\}$ is connected by an edge in $K$, the pairs $\{v_i,v_j\} \subseteq \{v_1,\cdots,v_k\}$ are connected by an edge in $K(K_1,\cdots,K_m)$, and so the set $\{v_1,\cdots,v_k\}$ is pairwise connected. Therefore, we have found a collection of pairwise connected vertices in  $K(K_1,\cdots,K_m)$ that does not form a face, and so $K(K_1,\cdots,K_m)$ cannot be flag. 
    \end{proof}

\begin{lem}
    Let $K$ be a simplicial complex that is not flag. For at least one vertex $v$ in a minimal missing face of $K$, suppose that $K_v$, that is, the simplicial complex substituted into $K$ at vertex $v$, contains a 1-simplex. Then $K(K_1,\cdots,K_m)$ is not the $k$-skeleton of any flag complex.
\end{lem}
\begin{proof}
    As $K$ is not flag, it will have a minimal missing face on at least $3$ vertices. Reordering vertices if necessary, let $\{1,\cdots,j\}$ be a minimal missing face of $K$, and suppose that $K_j$ has at least one $1$ simplex. By Lemma \ref{mf}, $\{v_1,\cdots,v_j\}$ is a minimal missing face in $K(K_1,\cdots,K_m)$, and as the boundary of $(1,\cdots,j)$ is contained in $K$, the boundary of $(v_1,\cdots,v_j)$ is contained in $K(K_1,\cdots,K_m)$. As $\{v_1,\cdots,v_j\}$ is a minimal missing face of $K(K_1,\cdots,K_m)$, if $K(K_1,\cdots,K_m)$ was the $k$-skeleton of a flag complex, it could only be the $j-1$-skeleton of a flag complex. Therefore it could not contain any simplices of dimension greater than $j-1$. Let $v^1_j$ and $v^2_j$ denote two vertices of $K_j$ that are joined by a edge. The link of $j$ in $K$ will contain the simplex $(1,\cdots,j-2)$, and so by definition of the substitution operation, we have $(v_1,\cdots,v_{j-2},v^1_j,v^2_j) \in K(K_1,\cdots,K_m)$, a contradiction. Therefore $K(K_1,\cdots,K_m)$ cannot be the $k$-skeleton of any flag complex. 
\end{proof} 

We now consider the case when $K$ is flag. 

\begin{lem}\label{notflag1}
   Let $K$ be a flag complex on $[m]$ vertices. If at least one $K_i$ substituted into $K$ is not flag, then $K(K_1,\cdots,K_m)$ is not flag.
\end{lem}
\begin{proof}
    By Lemma \ref{mf}, all missing faces of each $K_i$ are missing faces of $K(K_1,\cdots,K_m)$. Therefore $K(K_1,\cdots,K_m)$ will contain all the missing faces of $K_i$, and so cannot be flag. 
\end{proof}

\begin{lem}\label{kskelmis}  If $K$ is the $k$-skeleton of a flag complex, and $K \neq \Delta^k$, it must have a minimal missing face of dimension $k+1$.\end{lem}\begin{proof}  As $K$ is the $k$-skeleton of a flag complex, it must have at least one face of dimension $k$, and no higher faces.     Let $K^f$ denote the minimal flag complex on the same 1-skeleton as $K$, so $K$ is the $k$-skeleton of $K^f$.  Up to a reordering of vertices, let $(1,\cdots,k+2)$ be a $(k+1)$-simplex of $K^f$. In $K$, the set $\{1,\cdots,k+2\}$ does not form a simplex, as the dimension is greater than $k$, but the boundary is contained in $K$, because each $k$ simplex in the boundary is in the $k$-skeleton of $K^f$, which is $K$. Therefore  $\{1,\cdots,k+2\}$ is a minimal missing face of $K$. \end{proof}

\begin{lem}\label{notflag2}
    Let $K$ be a flag complex. If a) at least one $K_i$ is not the $k$-skeleton of a flag complex  or b)  $K_i$ is the $k$-skeleton of a flag complex and the link of $i$ in $K$ is non-empty, then $K(K_1,\cdots,K_m)$ is not the $k$-skeleton of a flag complex. 
\end{lem}
\begin{proof}
    If at least one $K_i$ is not the $k$-skeleton of a flag complex, then by combining Lemma \ref{mf} and the argument from Lemma \ref{notflag1}, $K(K_1,\cdots,K_m)$ cannot be the $k$-skeleton of some flag complex. So now consider case b). As $K_i$ is the $k$-skeleton of a flag complex, it will have a minimal missing face on at least $k+2$ vertices, by Lemma \ref{kskelmis}. Reordering vertices if necessary, let $\{1,\cdots,k+2\}$ be a minimal missing face of $K_i$. Therefore, if $K(K_1,\cdots,K_m)$ is the skeleton of a flag complex, it must be a $k$-skeleton. By assumption, the link of $i$ is non empty, and so there exists a $\sigma$ such that $\sigma \cup \{i\} \in K$. By definition of substitution, for all $\tau \in K_i$ $\sigma \cup \tau \in K(K_1,\cdots,K_m)$. More specifically, as the boundary of $(1,\cdots,k+2)$ is contained in $K_i$, $(1,\cdots,k+1) \cup \sigma$ is a simplex in $K(K_1,\cdots,K_m)$. But then we have a face of $K(K_1,\cdots,K_m)$ of a dimension that is strictly larger than $k$, and so $K(K_1,\cdots,K_m)$ is not the $k$-skeleton of a flag complex. 
\end{proof} 

We end this section by building some substitution complexes such that $\Omega \caa^{K(K_1,\cdots,K_m)} \in \cal{P}$, but $\caa^{K(K_1,\cdots,K_m)} \notin \cal{W}$ and $ K(K_1,\cdots,K_m)$ is not the $k$-skeleton of a flag complex. We do this by substituting into the boundary of a $m$-simplex. 

\begin{lem}
    There are simplicial equivalences  $(\partial\Delta^{l-1} \setminus \{l\})(K_{1},\cdots,K_{l-1},l)$ and  $\Delta^{l-2}(K_1,\cdots,K_{l-1})$ and $\mathrm{lk}_{\partial\Delta^{l-1}(K_{1},\cdots,K_{l-1},l)}(l)$ and $\partial \Delta^{l-2}(K_1,\cdots,K_{l-1})$.  
\end{lem}

\begin{proof}
These follows from Lemma \ref{deleteeq}, where $\kl = (\underline{K},\underline{\emptyset})$.
\end{proof}
We now have all the ingredients in place to produce new examples of spaces in $\cal{P}$.

\begin{thm}\label{new1}
    Let $\partial \Delta (K_1,\cdots, K_n)$ be a substitution complex such that for each $K_i$, we have $\Omega \caa^{K_i} \in \cal{P}$. Furthermore, suppose least one $K_i$ is such that $\caa^{K_i} \notin \cal{W}$. Then we have $\caa^{\partial \Delta (K_1,\cdots, K_n)} \notin \cal{W}$,  $\partial \Delta(K_1,\cdots,K_n)$ is not the $k$-skeleton of a flag complex for any $k$, and $\Omega \caa^{\partial \Delta (K_1,\cdots, K_n)} \in \cal{P}$.
\end{thm}
\begin{proof}
     As for all $i \in [m]$, we have $\Omega \caa^{K_i} \in \cal{P}$, we obtain $\Omega \caa^{\partial \Delta (K_1,\cdots, K_n)} \in \cal{P}$ by Corollary \ref{subpreserve}. As there exists an $i$ such that $\caa^{K_i} \notin \cal{W}$, by Lemma \ref{notwedge1}, we obtain that $\caa^{\partial \Delta (K_1,\cdots, K_n)} \notin \cal{W}$. As $\caa^{K_i} \notin \cal{W}$, we have that $K_i$ must have at least one edge, as otherwise $K_i$ would be the disjoint union of points, and \cite{MR2321037} established that if $K_i$ is of this form, then $\caa^{K_i} \in \cal{W}$.  As $\partial \Delta^{n-1}$ is not flag, Theorem \ref{notflag} implies $\partial \Delta (K_1,\cdots, K_n)$ is not the $k$-skeleton of a flag complex. 
\end{proof}
To conclude this section we give a concrete examples of a new complexes such that, the moment-angle complex $\cal{Z}_{K(K_1,\cdots K_m)} \in \cal{P}$.
\begin{ex}
  Let $K = \partial \Delta^{n-1}$, and let each $K_i$ be the boundary of an $n_i$-gon, denoted $P_{n_i}$, where $n_i \geq 4$. Observe that $\cal{Z}_{P_4} \simeq S^3 \times S^3$, and that for $m > 4$ ,by \cite{MR0531977} we have $\cal{Z}_{P_{n}} \cong \#_{k=3}^{n-1}(S^k \times S^{n+2-k})^{\#(k-2){n-1 \choose k-1}}$. Therefore, for $n_i \geq 4$ we have that $\cal{Z}_{P_{n_i}} \notin \cal{W}$. As $\Omega \cal{Z}_{P_n} \simeq \Omega S^3 \times \Omega S^{n-1} \times \Omega S(P_m)$ \cite{MR3228428}, where $S(P_n)$ is a wedge of simply connected spheres, we have $\Omega \cal{Z}_{P_n} \in \cal{P}$. By Theorem \ref{new1}, $\cal{Z}_{\partial \Delta^{m-1}(P_{n_1},\cdots,P_{n_m})} \in \cal{P}$. 
\end{ex}
\bibliographystyle{alpha}
\bibliography{bib}

\end{document}